\documentclass[preprint,12pt]{elsarticle}
\usepackage{graphicx}
\usepackage{amssymb}
\usepackage{algcompatible}
\usepackage{algpseudocode}
\usepackage{multirow}
\usepackage{multirow}
\usepackage{xspace}
\usepackage{hyperref}
\usepackage{listings}
\usepackage{url}
\usepackage{csquotes}
\usepackage[mathscr]{euscript} 
\usepackage{amsfonts,epsfig}
\usepackage{enumerate}
\usepackage{amsmath,amscd,amsfonts,amssymb}

\usepackage{lineno,hyperref}
\usepackage{graphicx,subfigure}
\usepackage{amsfonts,epsfig}

\usepackage[mathscr]{euscript} 
\usepackage{enumerate}
\usepackage{amsmath,amscd,amsfonts,amssymb}
\usepackage{times}
\usepackage{extarrows}
\usepackage{stmaryrd}
\usepackage{textcomp}
\usepackage[mathscr]{euscript}
\usepackage{algorithm2e}
\usepackage{colortbl}
\usepackage{color,soul}
\usepackage{newunicodechar}
\usepackage{amsthm}
\theoremstyle{definition}
\newtheorem{thm}{Theorem}

\newtheorem{defn}{Definition}
\newtheorem{exa}{Example}
\newtheorem{rem}[thm]{\bf{Remark}}
\newcommand{\cX}{\underline{\bf X}}
\newcommand{\cQ}{\underline{\bf Q}}
\newcommand{\cC}{\underline{\bf C}}
\newcommand{\cY}{\underline{\bf Y}}

\newcommand{\cU}{\underline{\bf U}}
\newcommand{\cV}{\underline{\bf V}}
\newcommand{\cZ}{\underline{\bf Z}}
\newcommand{\cR}{\underline{\bf R}}
\newcommand{\cS}{\underline{\bf S}}
\newcommand{\cP}{\underline{\bf P}}
\newcommand{\cW}{\underline{\bf W}}
\newcommand{\cD}{\underline{\bf D}}
\newcommand{\cL}{\underline{\bf L}}

\newcommand{\R}{{\bf R}}
\newcommand{\X}{{\bf X}}
\newcommand{\Y}{{\bf Y}}
\newcommand{\I}{{\bf I}}

\newcommand{\C}{{\bf C}}
\newcommand{\Q}{{\bf Q}}

\newcommand{\Z}{{\bf Z}}

\newcommand{\U}{{\bf U}}

\newcommand{\V}{{\bf V}}

\usepackage{amssymb}

\journal{Elsevier}

\begin{document}

\begin{frontmatter}

\title{A note on generalized tensor CUR approximation for tensor pairs and tensor triplets based on the tubal product}

 \author[label1]{Salman Ahmadi-Asl}
 \affiliation[label1]{organization={Innopolis University, Innopolis, Russia,\,s.ahmadiasl@innopolis.ru},
             }
\author[label2]{Naeim Rezaeian}
\affiliation[label2]{organization={Peoples' Friendship University of Russia, Moscow, Russia},}

\author[label3]{Keivan Ramazani}
\affiliation[label3]{organization={Razi University, Kermanshah, Iran},
}

\begin{abstract}
In this note, we briefly present a generalized tensor CUR (GTCUR) approximation for tensor pairs $(\cX,\cY)$ and tensor triplets $(\cX,\cY,\cZ)$ based on the tubal product (t-product). 
We use the tensor Discrete Empirical Interpolation Method (TDEIM) to do these extensions. We demonstrate how the TDEIM can be applied to extend the traditional tensor CUR (TCUR) approximation, which operates on a single tensor, to simultaneously compute the TCUR approximations for two or three tensors. This method allows for the sampling of relevant  lateral or horizontal slices from one data tensor in relation to one or two other data tensors. In certain special cases, the Generalized TCUR (GTCUR) method simplifies to the classical TCUR approximations for both tensor pairs and tensor triplets, akin to the process shown for matrices.
\end{abstract}

\end{frontmatter}

\section{Introduction}
The matrix CUR (MCUR) approximation is a well-known technique for fast low-rank approximation of matrices \cite{goreinov1997theory}. It offers an interpretable approximation and a compact data representation as it uses the actual columns/rows of the underlying data matrix to build the factor matrices.
 The generalized MCUR (GMCUR) approximation for matrix pairs and  matrix triplets were presented in \cite{gidisu2022generalized} and \cite{gidisu2022restricted}, respectively. Like the SVD, the MCUR is utilized for a single matrix, whereas the GMCUR addresses a pair or triplet of matrices, offering low-rank approximations for two or three matrices at the same time. The central concept behind these generalizations is to employ the same column or row indices for the two data matrices in their low-rank approximations. To be more precise, let two data matrices $\X\in\mathbb{R}^{I_1\times I_2},\,(I_1\geq I_2),\,\Y\in\mathbb{R}^{I_3\times I_2},\,(I_3\geq I_2)$ that have the same number of columns be given. The generalized SVD (GSVD) \cite{van1976generalizing,paige1981towards}, guarantees the following decompositions
\begin{eqnarray}\label{gs_1}
    \X&=&\U\,{\rm diag}(\alpha_1,\cdots,\alpha_{I_2})\,\Z,\quad \alpha_i\in[0,1]\\\label{gs_2}
    \Y&=&\V\,{\rm diag}(\beta_1,\cdots,\beta_{I_2})\,\Z,\quad\beta_i\in[0,1] 
\end{eqnarray}
where $\alpha_i^2+\beta_i^2=1$ with the ratios $\beta_i/\alpha_i$ of nondecreasing order for $i=1,2,\ldots,I_2$. Here, $\U\in\mathbb{R}^{I_1\times I_1},\,\V\in\mathbb{R}^{I_3\times I_3}$ are orthogonal, while $\Z\in\mathbb{R}^{I_2\times I_2}$ is a nonsingular matrix. As can be seen, the GSVD shares the common factor matrix $\Z$ between the given data matrices $\X$ and $\Y$. On the other hand, from the theory of the MCUR approximation, we know that by applying the Discrete Empirical Interpolation Method (DEIM) \cite{sorensen2016deim} to the basis $\Z$ (to be discussed in Section \ref{Sec:MACA}), we can sample appropriate column indices from the data matrices $\X$ and $\Y$. Since the same basis matrix $\Z$ is used within the DEIM, the same indices will be produced for sampling the columns of the matrix $\X$ and the matrix $\Y$. Besides, the matrices $\U$ and $\V$ can be used to sample row indices of the matrices $\X$ and $\Y$, respectively. This means that the selected row indices for these matrices can be different. As a result, for two matrices $\X$ and $\Y$ with the same number of columns, we can use the identical column indices for their MCUR approximations. In \cite{gidisu2022generalized} , the authors demonstrate that this concept can be effectively applied in scenarios where the goal is to identify the most distinguishing features from one dataset in relation to another. It is also relevant for recovering data matrices affected by colored noise and for subgroup discovery applications. Building on this idea, the authors in \cite{gidisu2022restricted} examine triplet matrices $({\bf X},{\bf Y},{\bf Z})$ and employ restricted SVD \cite{zha1991restricted} to construct their common factor matrices. They then apply the DEIM to these common factor matrices to sample both columns and rows. This method has been shown to effectively extract the most discriminative features from one dataset compared to two others. The GMCUR for matrix triplets is also inspired by canonical correlation analysis and has proven successful in recovering datasets influenced by colored noise; further details can be found in \cite{gidisu2022generalized,gidisu2022restricted}. 

In this note, due to the applications of tensors \cite{elden2019solving,asante2021matrix,asante2023image}, we focus on the tensor SVD (t-SVD) \cite{kilmer2013third} that is defined based on the tubal product (t-product). The t-SVD has similar properties as the classical SVD. In particular, in contrast to the Tucker decomposition \cite{tucker1964extension,tucker1966some} or Canonical Polyadic Decomposition \cite{hitchcock1927expression,hitchcock1928multiple}, its truncation provides the best low tubal rank approximation in the least-squares sense. The GSVD was generalized to tensors based on the t-SVD in \cite{he2021generalized,zhang2021cs}.
Inspired by the works in \cite{gidisu2022generalized,gidisu2022restricted}, we extend the tensor CUR (TCUR) approximation \cite{tarzanagh2018fast} to accommodate pairs and triplets of tensors based on the t-product, which we refer to as the generalized TCUR (GTCUR) approximation. Given that real-world datasets frequently exhibit multidimensional structures, it is beneficial to adapt the concepts introduced in \cite{gidisu2022generalized,gidisu2022hybrid} for tensors, ensuring that the integrity of the data tensor structures is maintained. Our objective is to achieve this goal, and we present several generalizations of GMCUR for tensors. To facilitate this, we utilize the tensor DEIM (TDEIM) index selection algorithm \cite{ahmadi2024robust}, which we recently developed to sample significant lateral and horizontal slices, playing a crucial role in the GTCUR method's development. Like the matrix case, the TDEIM facilitates the computation of identical lateral/horizontal slices for a specified pair or triplet of third-order tensors. It is worth noting that alternative sampling methods, such as top tubal leverage scores \cite{mahoney2011randomized,tarzanagh2018fast} or their sampling variants, could be used instead of the TDEIM. However, since the TDEIM has been demonstrated to provide the optimal sampling method \cite{sorensen2016deim,ahmadi2024robust}, we have chosen to adopt it in our research.

The remainder of this note is organized as follows. We first present some basic tensor concepts in Section \ref{Sec:prelim}. The matrix cross approximation is outlined in Section \ref{Sec:MACA} and the Tensor Discrete Empirical Interpolation Method (TDEIM) as an extension of the DEIM method is presented in Section \ref{deimt}. In Sections \ref{Sec:PTACA} and \ref{Sec:PTACT}, we show how the matrix cross approximation for matrix pairs/triplets can be generalized to tensor pairs and tensor triplets, respectively. The numerical experiments are presented in Section \ref{Sec:num}. The conclusion is given in Section \ref{Sec:Con}.

\section{Preliminaries}\label{Sec:prelim}
Prior to introducing the main materials, we will first outline the fundamental notations and definitions. An underlined bold capital case letter, a bold capital case letter, and a bold lower case letter are used to represent a tensor, a matrix, and a vector, respectively.  For a third-order tensor $\underline{\X},$ we call the slices $\underline{\X}(:,:,k),\,\underline{\X}(:,j,:),\,\underline{\X}(i,:,:)$ frontal, lateral, and horizontal slices. For simplicity, sometimes for a frontal slice $\cX(:,:,i)$, we use the notation $\cX_i$. 
For a third-order tensor $\cX,$ 
the fiber $\underline{\X}(i,j,:)$ is called a tube. {The matrix's component-wise complex conjugate is denoted by the symbol ``conj''.} To indicate a subset of a matrix or tensor, we use the MATLAB notations. For instance, for a given data matrix $\X$, by $\X(:,\mathcal{J})$ and $\X(\mathcal{I},:)$ we mean two matrices sampling a part of the rows and columns of the matrix $\X$, respectively, where $\mathcal{I}\subset\{1,2,\ldots,I_2\}$ and $\mathcal{J}\subset\{1,2,\ldots,I_1\}$. {For example, let
\[A=\begin{bmatrix}
 1 & 2 & 3 & 4\\
 5 & 6 & 7 & 8\\
 9 & 10 &11 & 12\\
 13 & 14 & 15 & 16\\
 17&18&19&20
\end{bmatrix}\in\mathbb{R}^{5\times 4}
\]
and $\mathcal{I}=\{1,2\}\subset\{1,2,3,4,5\},\, \mathcal{J}=\{2,4\}\subset\{1,2,3,4\}$, then 
\[
A(\mathcal{I},\mathcal{J})=\begin{bmatrix}
 2 & 4\\
 6 & 8
\end{bmatrix}.
\]
We will employ the same notation when sampling indices of tensor modes. }

To introduce the tensor SVD (t-SVD) model, we must first present Definitions 1-5.

\begin{defn} ({t-product})
Let $\underline{\mathbf X}\in\mathbb{R}^{I_1\times I_2\times I_3}$ and $\underline{\mathbf Y}\in\mathbb{R}^{I_2\times I_4\times I_3}$, the t-product $\underline{\mathbf X}*\underline{\mathbf Y}\in\mathbb{R}^{I_1\times I_4\times I_3}$ is defined as follows
\begin{equation}\label{TPROD}
\underline{\mathbf C} = \underline{\mathbf X} * \underline{\mathbf Y} = {\rm fold}\left( {{\rm circ}\left( \underline{\mathbf X} \right){\rm unfold}\left( \underline{\mathbf Y} \right)} \right),
\end{equation}
where 
\[
{\rm circ} \left(\underline{\mathbf X}\right)
=
\begin{bmatrix}
\underline{\mathbf X}(:,:,1) & \underline{\mathbf X}(:,:,I_3) & \cdots & \underline{\mathbf X}(:,:,2)\\
\underline{\mathbf X}(:,:,2) & \underline{\mathbf X}(:,:,1) & \cdots & \underline{\mathbf X}(:,:,3)\\
 \vdots & \vdots & \ddots &  \vdots \\
 \underline{\mathbf X}(:,:,I_3) & \underline{\mathbf X}(:,:,I_3-1) & \cdots & \underline{\mathbf X}(:,:,1)
\end{bmatrix},
\]
and 
\[
{\rm unfold}(\underline{\mathbf Y})=
\begin{bmatrix}
\underline{\mathbf Y}(:,:,1)\\
\underline{\mathbf Y}(:,:,2)\\
\vdots\\
\underline{\mathbf Y}(:,:,I_3)
\end{bmatrix},\hspace*{.5cm}
\underline{\mathbf Y}={\rm fold} \left({\rm unfold}\left(\underline{\mathbf Y}\right)\right).
\]
The Discreet Fourier Transform (DFT) is used to execute the t-product as described in \cite{kilmer2011factorization,kilmer2013third}, and it was recommended in \cite{kernfeld2015tensor} to use any invertible transformation instead of the DFT. Later, in \cite{jiang2020framelet} and \cite{li2022nonlinear}, nonivertible and even nonlinear mapping were employed. The ability to compute the t-SVD of a data tensor with a lower tubal rank is a benefit of using such unitary transformations \cite{song2020robust,jiang2020framelet}.
{The MATLAB command ${\rm fft}(\cX,[],3)$, computes the DFT of every tube of the data tensor $\cX$.} The fast version of the t-product in which the DFT of only the first $\lceil \frac{I_3+1}{2}\rceil$ frontal slices is needed is summarized in Algorithm \ref{ALG:TSVDP}, see \cite{hao2013facial,lu2019tensor} for details.
\end{defn}

\begin{defn} {\rm ({Transpose})
Let $\underline{\mathbf X}\in\mathbb{R}^{I_1\times I_2\times I_3}$ be a given tensor. Then the transpose of the tensor $\underline{\mathbf X}$ is denoted by $\underline{\mathbf X}^{T}\in\mathbb{R}^{I_2\times I_1\times I_3}$, which is constructed by transposing all its frontal slices and then reversing the order of transposed frontal
slices $2$ through $I_3$.}
\end{defn}

\begin{defn} ({Identity tensor})
Identity tensor $\underline{\mathbf I}\in\mathbb{R}^{I_1\times I_1\times I_3}$ is a tensor whose first frontal slice is an identity matrix of size $I_1\times I_1$ and all other frontal slices are zero. It is easy to show $\underline{\I}*\underline{\X}=\underline{\X}$ and $\underline{\X}*\underline{\I} =\underline{\X}$ for all tensors of conforming sizes.
\end{defn}
\begin{defn} ({Orthogonal tensor})
A tensor $\underline{\mathbf X}\in\mathbb{R}^{I_1\times I_1\times I_3}$ is orthogonal if ${\underline{\mathbf X}^T} * \underline{\mathbf X} = \underline{\mathbf X} * {\underline{\mathbf X}^ T} = \underline{\mathbf I}$.
\end{defn}

\begin{defn} ({f-diagonal tensor})
If all frontal slices of a tensor are diagonal, then the tensor is called f-diagonal.
\end{defn}

\begin{defn}
(Moore–Penrose pseudoinverse of a tensor) Let $\cX\in\mathbb{R}^{I_1\times I_2\times I_3}$ be given. The Moore-Penrose (MP) pseudoinverse of the tensor $\cX$ is denoted by $\cX^{\dag}\in\mathbb{R}^{I_2\times I_1\times I_3}$ is a unique tensor satisfying the following four equations:
\begin{eqnarray*}
\cX^{\dag}*\cX*\cX^{\dag}=\cX^{\dag},\quad \cX*\cX^{\dag}*\cX=\cX,\\
(\cX*\cX^{\dag})^T=\cX*\cX^{\dag},\quad (\cX^{\dag}*\cX)^T=\cX^{\dag}*\cX.
\end{eqnarray*}
The MP pseudoinverse of a tensor can also be computed in the Fourier domain and this is shown in Algorithm \ref{ALG:TSVDP}. \textcolor{blue}{The inverse of a tensor $\cX$, denoted by $\cX^{-1}$ is a special case of MP for which we have $\cX^{-1}*\cX=\cX*\cX^{-1}=\underline{\bf I}$.}
\end{defn}

\RestyleAlgo{ruled}
\LinesNumbered
\begin{algorithm}
\SetKwInOut{Input}{Input}
\SetKwInOut{Output}{Output}\Input{Two data tensors $\underline{\mathbf X} \in {\mathbb{R}^{{I_1} \times {I_2} \times {I_3}}},\,\,\underline{\mathbf Y} \in {\mathbb{R}^{{I_2} \times {I_4} \times {I_3}}}$} 
\Output{t-product $\underline{\mathbf C} = \underline{\mathbf X} * \underline{\mathbf Y}\in\mathbb{R}^{I_1\times I_4\times I_3}$}
\caption{Fast t-product of two tensors \cite{kilmer2011factorization,lu2019tensor}}\label{ALG:TSVDP}
      {
      $\widehat{\underline{\mathbf X}} = {\rm fft}\left( {\underline{\mathbf X},[],3} \right)$;\\
      $\widehat{\underline{\mathbf Y}} = {\rm fft}\left( {\underline{\mathbf Y},[],3} \right)$;\\
\For{$i=1,2,\ldots,\lceil \frac{I_3+1}{2}\rceil$}
{                        
$\widehat{\underline{\mathbf C}}\left( {:,:,i} \right) = \widehat{\underline{\mathbf X}}\left( {:,:,i} \right)\,\widehat{\underline{\mathbf Y}}\left( {:,:,i} \right)$;\\
}
\For{$i=\lceil\frac{I_3+1}{2}\rceil+1\ldots,I_3$}{
$\widehat{\underline{\mathbf C}}\left( {:,:,i} \right)={\rm conj}(\widehat{\underline{\mathbf C}}\left( {:,:,I_3-i+2} \right))$;
}
$\underline{\mathbf C} = {\rm ifft}\left( {\widehat{\underline{\mathbf C}},[],3} \right)$;   
       	}       	
\end{algorithm}

\RestyleAlgo{ruled}
\LinesNumbered
\begin{algorithm}
\SetKwInOut{Input}{Input}
\SetKwInOut{Output}{Output}\Input{The data tensor $\underline{\mathbf X} \in {\mathbb{R}^{{I_1} \times {I_2} \times {I_3}}}$} 
\Output{Moore-Penrose pseudoinverse $\underline{\mathbf X}^{\dag}\in\mathbb{R}^{I_2\times I_1\times I_3}$}
\caption{Fast Moore-Penrose pseudoinverse computation of the tensor $\underline{\bf X}$}\label{ALG:TSVDP}
      {
      $\widehat{\underline{\mathbf X}} = {\rm fft}\left( {\underline{\mathbf X},[],3} \right)$;\\
\For{$i=1,2,\ldots,\lceil \frac{I_3+1}{2}\rceil$}
{                        
$\widehat{\underline{\mathbf C}}\left( {:,:,i} \right) = {\rm pinv}\,(\widehat{\X}(:,:,i))$;\\
}
\For{$i=\lceil\frac{I_3+1}{2}\rceil+1,\ldots,I_3$}{
$\widehat{\underline{\mathbf C}}\left( {:,:,i} \right)={\rm conj}(\widehat{\underline{\mathbf C}}\left( {:,:,I_3-i+2} \right))$;
}
$\underline{\mathbf X}^{\dag} = {\rm ifft}\left( {\widehat{\underline{\mathbf C}},[],3} \right)$;   
       	}       	
\end{algorithm}

\textcolor{blue}{The tensor SVD (t-SVD) is a viable tensor decomposition that represents a tensor as the t-product of three tensors. The first and last tensors are orthogonal, while the middle tensor is an f-diagonal tensor. Let $\underline{\X}\in\mathbb{R}^{I_1\times I_2\times I_3}$, then the t-SVD gives the following model:
\[
\underline{\X}= \cU*\underline{\bf S}*\cV^T,
\]
where $\cU\in\mathbb{R}^{I_1\times I_1\times I_3 },\,\underline{\bf S}\in\mathbb{R}^{I_1\times I_2\times I_3},$ and $\cV\in\mathbb{R}^{I_2\times I_2 \times I_3}$. The tensors $\cU$ and $\cV$ are orthogonal, while the tensor $\underline{\bf S}$ is f-diagonal. The procedure of the computation of the t-SVD is presented in Algorithm \ref{ALG:Tsvd}.  A truncated T-SVD of tubal rank $R$ is obtained as follows:
\[
\underline{\X}\approx \cU_R*\underline{\bf S}_R*\cV_R^T,
\]
where $\cU_R=\cU(:,1:R,:)\in\mathbb{R}^{I_1\times R\times I_3},\,\cV_R=\cV(:,1:R,:)\in\mathbb{R}^{I_2\times R\times I_3}$ and $\underline{\bf S}_R=\underline{\bf S}(1:R,1:R,:)\in\mathbb{R}^{R\times R\times I_3}$.}

\RestyleAlgo{ruled}
\LinesNumbered
\begin{algorithm}
\SetKwInOut{Input}{Input}
\SetKwInOut{Output}{Output}\Input{The data tensor $\underline{\mathbf X} \in {\mathbb{R}^{{I_1} \times {I_2} \times {I_3}}}$ and a target tubal-rank $R$} 
\Output{The truncated t-SVD of the tensor $\cX$ as $\underline{\bf X}\approx\cU*\underline{\bf S}*\cV^T$}
\caption{The truncated t-SVD decomposition of the tensor}\label{ALG:Tsvd}
      {
      $\widehat{\underline{\mathbf X}} = {\rm fft}\left( {\underline{\mathbf X},[],3} \right)$;\\
\For{$i=1,2,\ldots,\lceil \frac{I_3+1}{2}\rceil$}
{                        
$[\widehat{\underline{\mathbf U}}\left( {:,:,i} \right),\widehat{\underline{\bf  S}}(:,:,i),\widehat{\V}(:,:,i)] = {\rm Truncated\operatorname{-} svd}\,(\widehat{\X}(:,:,i),R)$;\\
}
\For{$i=\lceil\frac{I_3+1}{2}\rceil+1,\ldots,I_3$}{
$\widehat{\underline{\mathbf U}}\left( {:,:,i} \right)={\rm conj}(\widehat{\underline{\mathbf U}}\left( {:,:,I_3-i+2} \right))$;\\
$\widehat{\underline{\mathbf S}}\left( {:,:,i} \right)=\widehat{\underline{\mathbf S}}\left( {:,:,I_3-i+2} \right)$;\\
$\widehat{\underline{\mathbf V}}\left( {:,:,i} \right)={\rm conj}(\widehat{\underline{\mathbf V}}\left( {:,:,I_3-i+2} \right))$;
}
$\underline{\mathbf U}_R= {\rm ifft}\left( {\widehat{\underline{\mathbf U}},[],3} \right)$;
$\underline{\mathbf S}_R= {\rm ifft}\left( {\widehat{\underline{\mathbf S}},[],3} \right)$; 
$\underline{\mathbf V}_R= {\rm ifft}\left( {\widehat{\underline{\mathbf V}},[],3} \right)$;
       	}       	
\end{algorithm}

\section{Matrix CUR (MCUR) approximation of a single matrix, matrix pairs and matrix triples}\label{Sec:MACA}
A well-known technique for fast low-rank approximation is the Matrix CUR (MCUR) method, or cross approximation \cite{goreinov1997theory}. Its strength lies in selecting a small subset of the original matrix's columns and rows to form an interpretable factorization. This approach maintains the matrix's salient features, e.g. like sparsity and non-negativity,  at linear computational cost.  Let us describe it formally. For a given data matrix $\X\in\mathbb{R}^{I_1\times I_2}$ , the MCUR method seeks the approximation of the form $\X\approx\C\U\R$ where $\C=\X(:,\mathcal{J}),\,\R=\X(\mathcal{I},:)$ and $\mathcal{I}\subset \{1,2,\ldots,I_1\}$ and $\mathcal{J}\subset\{1,2,\ldots,I_2\}$. It is obvious that the best middle matrix is 
\begin{eqnarray}\label{midmatrix}
\U=\C^{\dag}\X{\R}^{\dag},
\end{eqnarray}
as it minimizes the error term or
\begin{eqnarray}
 \min_{\U} \|\X-\C\U\R\|^2_F. 
\end{eqnarray}
It is possible to select columns and rows either deterministically or randomly, allowing for the option to achieve either additive or relative approximation errors. It is widely recognized that in a deterministic context, selecting columns or rows with the highest volume can lead to nearly optimal solutions \cite{goreinov2010find}. The DEIM is another deterministic approach for selecting a matrix's columns or rows, which is based on the leading singular vectors \cite{sorensen2016deim}.. The three most commonly utilized probability distributions for column and row selection are uniform, length-squared, and leverage-score distributions. For further information on these sampling methods, refer to \cite{mahoney2011randomized}. Research has demonstrated that using the leverage-score probability distribution for sampling columns can yield more practical approximations with relative error accuracy \cite{drineas2008relative}.

The MCUR approximation for matrix pairs (matrices that have the same number of columns) was introduced as an extension of classical MCUR approximations in \cite{gidisu2022generalized}. The key idea of this generalization involves selecting columns from two matrices that share identical column indices to identify the features of one dataset that are most relevant to the other. In this approach, the common factor matrix obtained through the Generalized Singular Value Decomposition (GSVD) of the two matrices is utilized to sample the columns using the DEIM technique. Additionally, the MCUR approximation for triplet matrices was discussed in \cite{gidisu2022restricted}, aiming to uncover relevant characteristics of one data matrix in relation to two others. In this case, the authors employed the Randomized Singular Value Decomposition (RSVD) to generate two common factor matrices, after which DEIM is applied to sample both column and row indices. We will first explore the extension of DEIM to tensors based on the t-product in the next section, then we use it to extend the TCUR approximation to tensor pairs and tensor triplets in Sections \ref{Sec:PTACA} and \ref{Sec:PTACT}, respectively.

\section{Tensor Discrete Empirical Interpolation Method (TDEIM) for lateral/horizontal slice sampling}\label{deimt}
TCUR approximates a tensor $\cX\in\mathbb{R}^{I_1\times I_2\times I_3}$ as $\cX\approx\cX(:,\mathcal{I},:)*\cU*\cX(\mathcal{J},:,:)$, where $\cX(:,\mathcal{I},:)$ and $\cX(\mathcal{J},:,:)$ are actual lateral and horizontal slices of the original tensor and $\cU=\cX(:,\mathcal{I},:)^{\dagger}*\cX*\cX(\mathcal{J},:,:)^{\dagger}$, see Figure \ref{tcur_fig} for visual illustration and we refer to the review paper \cite{ahmadi2021cross} for more details on different types of TCUR methods. Determining a suitable selection of lateral and horizontal slices is a complex task. Recently, we extended the DEIM method to tensors using the t-product \cite{ahmadi2024robust}, which we refer to as tensor DEIM (TDEIM). Experimental results demonstrate that TDEIM provides superior sampling accuracy compared to existing methods, including top tubal leverage score sampling and uniform sampling. The process for using TDEIM to select horizontal slices is outlined in Algorithm 4.

Appropriate section of lateral and horizontal slices is a challenging problem. We have recently generalized the DEIM method to the tensor case based on the t-product \cite{ahmadi2024robust} that we call tensor DEIM (TDEIM). It has been experimentally shown that the TDEIM achieved the best sampling accuracy compared to the existing sampling techniques, such as top tubal leverage score sampling and uniform sampling.  The process for using TDEIM to select horizontal slices is outlined in Algorithm \ref{ALG:DEIM_ten}.
\begin{figure}
\begin{center}
\includegraphics[width=0.7\linewidth]{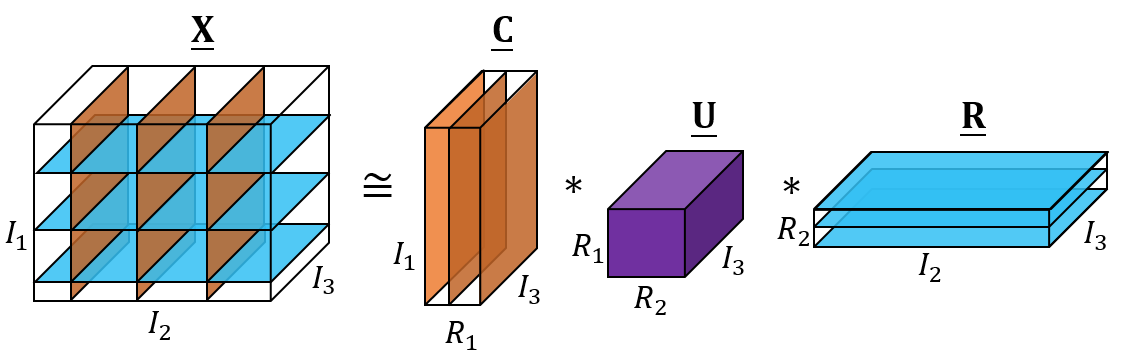}\\
\caption{\small{Tensor CUR approximation based on slice sampling \cite{tarzanagh2018fast,ahmadi2021cross}.}}\label{tcur_fig}
\end{center}
\end{figure}
A fundamental concept of the TDEIM is the {\it interpolatory projector}, which we will define now. For a given set of indices ${\bf s}\in\mathbb{R}^R$, the full tubal-rank tensor\footnote{A tensor with linear independent lateral slices, for example, the tensor $\cU$ obtained from the t-SVD can be used.} $\cU\in\mathbb{R}^{I_1\times R\times I_3}$ and considering $\underline{\bf S}=\underline{\bf I}(:,{\bf s},:)\in\mathbb{R}^{I_1\times R\times I_3}$, as the index selection tensor\footnote{It is called index selection tensor because $\underline{\bf X}({\bf s},:,:)=\underline{\bf S}^T*\underline{\bf X}.$}. where $\underline{\bf I}\in\mathbb{R}^{I_1\times I_1\times I_3}$ is an identity tensor, we build the tensorial oblique projection operator defined as follows
\begin{eqnarray}\label{proj}
\underline{\mathcal{P}}=\underline{\bf U}*(\underline{\bf{S}}^T*\underline{\bf{U}})^{-1}*\underline{\bf{S}}^T.
\end{eqnarray}
Given an arbitrary tensor $\underline{\bf G}\in\mathbb{R}^{I_1\times I_2\times I_3}$ and $\cX=\underline{\mathcal{P}}*\underline{\bf G}$, we have  
\begin{eqnarray}
\nonumber
\cX({\bf s},:,:)&=&\underline{\bf S}^T*\cX=
\underline{\bf S}^T*\underline{\bf U}*(\underline{\bf{S}}^T*\underline{\bf{U}})^{-1}*\underline{\bf{S}}^T*\underline{\bf G}\\
&=&
\underline{\bf{S}}^T*\underline{\bf G}=\underline{\bf G}({\bf s},:,:).
\end{eqnarray}
This means that the projection operator $\underline{\mathcal{P}}$ preserves the horizontal slices of $\underline{\bf G}$ specified by the index set ${\bf s}$. This justifies the name of interpolation. The TDEIM starts with selecting an index with the maximum Euclidean norm of the tubes in the first lateral slice of the basis, i.e. $\cU(:,1,:),$ and assign it as the index of the first sampled horizontal slice. The subsequent indices are selected according to the indices with the maximum Euclidean norm of the tubes of the residual lateral slice that is computed by removing the direction of the tensorial interpolatory projection in the previous basis vectors from the subsequent one. To be more precise, let the indices ${\bf s}_{j-1}=\{s_1,s_2,\ldots,s_{j-1}\}$ have been already selected, and we want to select $s_j$. To do so, we compute the residual slice
\[
\underline{\bf R}(:,j,:)=\underline{\bf U}(:,j,:)-\underline{\mathcal{P}}_{j-1}*\underline{\bf U}(:,j,:),
\]
where $\underline{\mathcal{P}}_{j-1}=\underline{\bf U}^{j-1}*({\underline{\bf S}^{j-1}}^T*\underline{{\bf U}}^{j-1})^{-1}*{\underline{\bf S}^{j-1}}^T,\,\underline{\bf S}^{j-1}=\underline{\bf I}(:,{\bf s}_{j-1},:),$ and $\underline{\bf U}^{j-1}=\underline{\bf U}(:,{\bf s}_{j-1},:)$. Then, a new index with the maximum Euclidean norm of tubes or
\[
s_j=\arg\max_{1\leq i \leq I_1}\|\underline{\bf R}(i,j,:)\|,
\]
is set as our new index $s_j$. This will be the new sampled horizontal slice's index. The same procedure can be used to select lateral slices where the tensor $\cV$ obtained from the t-SVD of the tensor $\cX$ should be used. 

Similar to the matrix case, let us introduce two quantities as follows
\begin{eqnarray}\label{t_const}
 \tilde{\eta}_p=\frac{1}{I_3}\max_{i}\,\left({\|(\widehat{\cS}^T_i\widehat{\cU}_i)^{-1}\|_2^2}\right),\quad\tilde{\eta}_q=\frac{1}{I_3}\max_{i}\,\left({\|(\widehat{\cV}^T_i\widehat{\cQ}_i)^{-1}\|_2^2}\right), 
\end{eqnarray}
 that will be used in the next theorem. Here, $\cV\in\mathbb{R}^{I_2\times R\times I_3}$ is a basis tensor for the subspace of horizontal slices of the original data tensor, $\underline{\bf Q}=\underline{\bf I}(:,{\bf s}_{j-1},:)$ is a tensor of some sampled lateral slices of the identity tensor specified by the index set ${\bf s}_{j-1}=\{s_1,s_2,\ldots,s_{j-1}\},\,\widehat{\cQ}={\rm fft}(\cQ,[],3),\,\widehat{\cV}={\rm fft}(\cV,[],3),\,\widehat{\cQ}_i=\widehat{\cQ}(:,:,i),$ and $\widehat{\cV}_i=\widehat{\cV}(:,:,i)$.   
 
The next theorem provides the upper error bound of the TCUR approximation obtained by the TDEIM sampling method.

\begin{thm}\label{thm_tub} \cite{ahmadi2024robust}
Suppose $\cX\in\mathbb{R}^{I_1\times I_2\times I_3}$ and $1\leq R<\min(I_1,I_2)$. Assume that  horizontal slice and lateral slice indices ${\bf p}$ and ${\bf q}$ give full tubal rank tensors $\cC=\cX(:,{\bf q},:)=\cX*\underline{\mathcal Q}$
and $\cR=\cX({\bf p},:,:)={\underline{\mathcal P}}*\cX$, where $\underline{\mathcal{P}}$ and $\underline{\mathcal{Q}}$ are tensorial interpolation projectors for horizontal and lateral slice sampling\footnote{The  tensorial
interpolation projector for the lateral slice sampling is defined as $\underline{\mathcal{Q}}=\underline{\V}*(\underline{\Q}^T*\underline{\V})^{-1}*\underline{\Q}^T$, where $\underline{\V}$ is a basis tensor and $\underline{\Q}$ is the index selection tensor. The same can be defined for horizontal slice selection where a basis tensor $\underline{\bf U}$ is used.}, respectively with corresponding finite error constants $\tilde{\eta}_p,\,\tilde{\eta}_q$ defined in \eqref{t_const} and set $\cU=\underline{\C}^{\dag}*\cX*\cR^{\dag}$. Then
\begin{eqnarray}
    \|\cX-\cC*\cU*\cR\|^2_F\leq (\tilde{\eta}_p+\tilde{\eta}_q)\sum_{i=1}^{I_3}\sum_{t>R}(\sigma^i_{t})^2,
\end{eqnarray} 
where $\sigma^i_{t}$ is the $t$-th largest singular values of the frontal slice $\widehat{\cX}_i=\widehat{\cX}(:,:,i)$.
\end{thm}

Theorem \ref{thm_tub} shows that the TCUR approximation with the middle tensor defined above can provide approximation within a factor $\tilde{\eta}_p+\tilde{\eta}_q$ of the best tubal rank $k$ approximation. It also indicates that the conditioning of the problem depends on these two quantities and the lateral/horizontal slices should be selected in such a way that these quantities should be controlled. In \cite{ahmadi2024robust}, it was empirically shown that these quantities are sufficiently small for many datasets. 

\begin{rem}
The basis tensors $\cU$ and $\cV$ required in Algorithm \ref{ALG:DEIM_ten} can be computed very fast through the randomized truncated t-SVD \cite{zhang2018randomized,ahmadi2022efficient,ahmadi2024randomized}. This version can be regarded as a randomized version of the TDEIM algorithm.
\end{rem}

\RestyleAlgo{ruled}
\LinesNumbered
\begin{algorithm}
\SetKwInOut{Input}{Input}
\SetKwInOut{Output}{Output}\Input{ $\cU\in\mathbb{R}^{I_1\times R\times I_3}$ with $R\leq I_1$ (linearly independent lateral slices)} 
\Output{Indices ${\bf s}\in\mathbb{N}^R$ with distinct entries in $\{1,2,\ldots,I_1\}$}
\caption{Tubal DEIM (TDEIM) index selection approach for horizontal slice selection}\label{ALG:DEIM_ten}
      {
$s_1=\arg\max_{1\leq i\leq I_1}\|\cU(i,1,:)\|_2$\\
\For{$j=2,3,\ldots,R$}{
$\underline{\C}=\cU({\bf s},1:j-1,:)^{-1}*\cU({\bf s},j,:)$;\\
$\underline\R=\cU(:,j,:)-\cU(:,1:j-1,:)*\underline{\C}$;\\
$s_j=\arg\max_{1\leq i\leq I_1}\|\underline{\R}(i,j,:)\|_2$
} 
       	}       	
\end{algorithm}

In the next sections, we extend the TCUR to tensor pairs and tensor triples. 

\section{Tensor CUR (TCUR) approximation for tensor pairs based on the t-product}\label{Sec:PTACA}
Inspired by the success of the MCUR approximation, several generalizations of these approaches to tensors have been proposed. For example, see \cite{oseledets2008tucker} for the Tucker model, see \cite{caiafa2010generalizing,oseledets2010tt} for the TT decomposition, see \cite{tarzanagh2018fast} for the t-SVD, and see \cite{espig2012note} for the TC/TR decomposition. In this note, we focus on the t-SVD case and extend it to a pair of tensors. The intuition behind this generalization is similar to the matrix case. More precisely, the indices of selected lateral slices are the same for two data tensors (see Figure \ref{fig:CP} for a graphical illustration). The formal definition of the GTCUR approximation for tensor pairs is presented below.
\begin{figure}
\centering    \includegraphics[width=.9\linewidth]{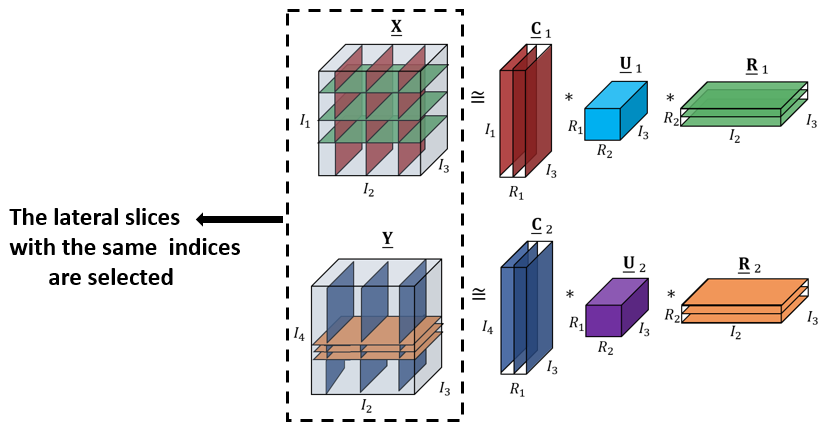}
    \caption{Visualization of the generalized tensor CUR (GTCUR) approximation for tensor pairs $(\cX,\cY)$. The indices of the selected lateral slices of the tensors $\cX$ and $\cY$ are identical while this is not necessarily true for the horizontal slices.}\label{fig:CP}
\end{figure}
\begin{defn}
Let $\cX\in\mathbb{R}^{I_1\times I_2 \times I_3}$ and $\cY\in\mathbb{R}^{I_4\times I_2\times I_3}$ be of full tubal rank, with $I_1 \geq I_2$ and $I_4 \geq I_2$.
A GTCUR decomposition of the tensor pairs $(\cX,\cY)$ with the tubal rank $R,$ is defined as follows
\begin{eqnarray}\label{pgtcur}
\cX\approx {\underline{\bf C}}_1*\cU_1*\cR_1=(\cX*\cP)*\cU_1*(\cS_1^T*\cX),\\\label{pgtcur_2}
\cY\approx {\underline{\bf C}}_2*\cU_2*\cR_2=(\cY*\cP)*\cU_2*(\cS_2^T*\cY),
\end{eqnarray}
where $\cS_1\in\mathbb{R}^{I_1\times R \times I_3},\,\cS_2\in\mathbb{R}^{I_4\times R \times I_3}$ and $\cP\in\mathbb{R}^{I_2\times R \times I_3}$ are the index selection tensors ($R<I_2$).
\end{defn}
Assume that ${\bf s}_1,$ and ${\bf s}_2$ are the horizontal slice indices sampled from the tensors $\cX$ and $\cY$, respectively. Also, let ${\bf p}$ be the common lateral slice indices for both tensors $\cX$ and $\cY$. Then, formulations \eqref{pgtcur}-\eqref{pgtcur_2} can be equivalently written as
\begin{eqnarray}\label{tgcur_i}
\cX&\approx &\cX(:,{\bf p},:)*\cU_1*\cX({\bf s}_1,:,:),\\\label{tgcur_ii}
\cY&\approx &\cY(:,{\bf p},:)*\cU_2*\cY({\bf s}_2,:,:).
\end{eqnarray}
The GTCUR approximation for tensor pairs indeed computes the TCUR approximations of two given data tensors $\cX$ and $\cY$ and is motivated by the generalized t-SVD (GTSVD). For given data tensors $\cX\in\mathbb{R}^{I_1\times I_2 \times I_3}$ and $\cY\in\mathbb{R}^{I_4\times I_2\times I_3}$, the GTSVD approximation  guarantees the following decompositions \cite{zhang2021cs}
\begin{eqnarray}\label{GTSVD_F}
\underline{\bf X}&=&\underline{\bf U}*\underline{\bf C}*\underline{\bf Z}^{T},\\\label{GTSVD_F2}
\underline{\bf Y}&=&\underline{\bf V}*\underline{\bf S}*\underline{\bf Z}^{T},
\end{eqnarray}
where $\underline{\bf U}\in\mathbb{R}^{I_1\times I_1\times I_3 },\,\underline{\bf V}\in\mathbb{R}^{I_4 \times I_4\times I_3 },\,\underline{\bf C}\in\mathbb{R}^{I_1\times I_2\times I_3 },\,\underline{\bf S}\in\mathbb{R}^{I_4\times I_2\times I_3 },\,\underline{\bf Z}\in\mathbb{R}^{I_2\times I_2\times I_3}$. Note that the tensors $\underline{\bf C}$ and $\underline{\bf S}$ are f-diagonal, and the tensors $\underline{\bf U}$ and $\underline{\bf V}$ are orthogonal, while the tensor $\underline{\bf Z}$ is nonsingular. The process of computing GTSVD is presented in Algorithm \ref{ALG:GTSVD}. We need to apply the classical GSVD (lines 3-5) to the first $\lceil \frac{I_3+1}{2}\rceil$ frontal slices of the tensors $\underline{\bf X}$ and $\underline{\bf Y}$ in the Fourier domain and the rest of the slices are computed easily (Lines 6-12). 
\RestyleAlgo{ruled}
\LinesNumbered
\begin{algorithm}
\SetKwInOut{Input}{Input}
\SetKwInOut{Output}{Output}\Input{The data tensors $\underline{\mathbf X} \in {\mathbb{R}^{{I_1} \times {I_2} \times {I_3}}}$ and $\underline{\mathbf Y} \in {\mathbb{R}^{{I_4} \times {I_2} \times {I_3}}}$} 
\Output{The generalized t-SVD of $\underline{\mathbf X}$ and $\underline{\bf Y}$ as $\underline{\mathbf X}=\underline{\mathbf U}*\underline{\mathbf C}*\underline{\mathbf Z}$ and $\underline{\mathbf Y}=\underline{\mathbf V}*\underline{\mathbf S}*\underline{\mathbf Z}$}
\caption{Generalized t-SVD of $\underline{\bf X}$ and $\underline{\bf Y}$}\label{ALG:GTSVD}
      {
      $\widehat{\underline{\mathbf X}} = {\rm fft}\left( {\underline{\mathbf X},[],3} \right)$;\\
        $\widehat{\underline{\mathbf Y}} = {\rm fft}\left( {\underline{\mathbf Y},[],3} \right)$;\\
\For{$i=1,2,\ldots,\lceil \frac{I_3+1}{2}\rceil$}
{                        
$[\widehat{\underline{\mathbf U}}_i,\,\widehat{\underline{\mathbf V}}_i,\,\widehat{\underline{\mathbf Z}}_i,\,\widehat{\underline{\mathbf C}}_i,\,\widehat{\underline{\mathbf S}}_i]= {\rm GSVD}\,(\widehat{\underline{\X}}(:,:,i),\widehat{\underline{\Y}}(:,:,i))$;\\
}
\For{$i=\lceil\frac{I_3+1}{2}\rceil+1,\ldots,I_3$}{
$\widehat{\underline{\mathbf U}}_i={\rm conj}(\widehat{\underline{\mathbf U}}_{I_3-i+2})$;\\
$\widehat{\underline{\mathbf V}}_i={\rm conj}({\underline{\mathbf V}}_{I_3-i+2})$;\\
$\widehat{\underline{\mathbf Z}}_i={\rm conj}(\widehat{\underline{\mathbf Z}}_{I_3-i+2})$;\\
$\widehat{\underline{\mathbf C}}_i=\widehat{\underline{\mathbf C}}_{I_3-i+2}$;\\
$\widehat{\underline{\mathbf S}}_i=\widehat{\underline{\mathbf S}}_{I_3-i+2}$;
}
${\underline{\mathbf U}} = {\rm ifft}\left( {\widehat{\underline{\mathbf U}},[],3} \right)$;
$\,{\underline{\mathbf V}} = {\rm ifft}\left( {\widehat{\underline{\mathbf V}},[],3} \right)$; 
$\,{\underline{\mathbf Z}} = {\rm ifft}\left( {\widehat{\underline{\mathbf Z}},[],3} \right)$;
$\,{\underline{\mathbf C}} = {\rm ifft}\left( {\widehat{\underline{\mathbf C}},[],3} \right)$; 
$\,{\underline{\mathbf S}} = {\rm ifft}\left( {\widehat{\underline{\mathbf S}},[],3} \right)$; 
       	}       	
\end{algorithm}
We see that the GTSVD provides a common right tensor $\cZ$ in \eqref{GTSVD_F}-\eqref{GTSVD_F2} and we can use it to sample lateral slices of the data tensors $\cX$ and $\cY$ based on the TDEIM algorithm. As a result, the same indices can be used to sample the lateral slices for the data tensors $\cX$ and $\cY$. The horizontal slices of the data tensors $\cX$ and $\cY$ can also be sampled using the left tensor parts $\cU$ and $\cV$, although they do not necessarily provide identical horizontal slice indices. Following this idea, we can compute the GTSVD of the tensors $\cX,\,\cY$ and by applying the TDEIM to the shared tensor factors $\cU,\,\cV,\,\cZ$, we can choose indices for the horizontal and lateral slices of the specified data tensors. This approach is summarized in Algorithm \ref{ALG:DEIM_genten}. In Line 1 of Algorithm \ref{ALG:DEIM_genten}, 
we need to calculate the GTSVD of two tensors, which can be quite challenging for large-scale tensors. However, to address this issue, the randomized algorithms proposed in \cite{ahmadi2025randomized} can be used. Note that Lines 10-11 in Algorithm \ref{ALG:DEIM_genten} can be efficiently computed, and this is outlined in Algorithm \ref{ALG:fast}. 
\RestyleAlgo{ruled}  
\LinesNumbered
\begin{algorithm}
\SetKwInOut{Input}{Input}
\SetKwInOut{Output}{Output}\Input{Two data tensors $\cX\in\mathbb{R}^{I_1\times I_2\times I_3}$ and $\cY\in\mathbb{R}^{I_4\times I_2\times I_3}$ where $I_1\geq I_2,\,I_4\geq I_2$ and a target tubal rank $R$} 
\Output{A rank-$R$ GTCUR approximation for tensor pairs $(\cX,\cY)$
\begin{eqnarray*}
\cX&\approx& \cX(:,{\bf p},:)*\underline{\U}_{1}*\cX({\bf s}_1,:,:),\\\quad\cY&\approx&\cY(:,{\bf p},:)*\underline{\U}_{2}*\cY({\bf s}_2,:,:)
\end{eqnarray*}
}
\caption{The GTCUR approximation for tensor pairs}\label{ALG:DEIM_genten}
      {
      $[\cU,\cV,\cZ]={\rm GTSVD}(\cX,\cY)$;\hspace{.2cm} (or randomized GTSVD \cite{ahmadi2025randomized}) \\
\For{$j=1,2,\ldots,R$}{
${\bf p}(j)={\arg\max}_{1\leq i\leq I_2}\|\cZ(i,j,:)\|$\\
${\bf s}_1(j)=\arg\max_{1\leq i\leq I_1}\|\cU(i,j,:)\|$\\
${\bf s}_2(j)=\arg\max_{1\leq i\leq I_4}\|\cV(i,j,:)\|$\\
$\cZ(:,j+1,:)=\cZ(:,j+1,:)-\cZ(:,1:j,:)*\cZ({\bf p},1:j,:)^{-1}*\cZ({\bf p},j+1,:)$\\
${\smaller \cU(:,j+1,:)=\cU(:,j+1,:)-\cU(:,1:j,:)*
\cU({\bf s}_1,1:j,:)^{-1}*\cU({\bf s}_1,j+1,:)}$\\
${\smaller \cV(:,j+1,:)=\cV(:,j+1,:)-\cV(:,1:j,:)*\cV({\bf s}_2,1:j,:)^{-1}*\cV({\bf s}_2,j+1,:)}$\\
} 
${\cU}_{1}=\cX(:,{\bf p},:)^{\dag}*\cX*\cX({\bf s}_1,:,:)^{\dag}$\\
${\cU}_{2}=\cY(:,{\bf p},:)^{\dag}*\cY*\cY({\bf s}_2,:,:)^{\dag}$
       	}       	
\end{algorithm}

\RestyleAlgo{ruled}
\LinesNumbered
\begin{algorithm}
\SetKwInOut{Input}{Input}
\SetKwInOut{Output}{Output}\Input{Two data tensors $\cX\in\mathbb{R}^{I_1\times I_2\times I_3}$ and $\cY\in\mathbb{R}^{I_4\times I_2\times I_3}$ where $(I_1\leq I_2,\,I_4\leq I_2)$, and the index sets ${\bf p},\,{\bf s}_1,\,{\bf s}_2$} 
\Output{Computation of the middle tensors in Lines 10-11 of Algorithm \ref{ALG:DEIM_genten} as 
\begin{eqnarray*}
{\cU}^{(1)}=\cX(:,{\bf p},:)^{\dag}*\cX*\cX({\bf s}_1,:,:)^{\dag}\\
{\cU}^{(2)}=\cY(:,{\bf p},:)^{\dag}*\cY*\cY({\bf s}_2,:,:)^{\dag}
\end{eqnarray*}
}
\caption{Fast computations of Lines 10-11 in Algorithm \ref{ALG:DEIM_genten} }\label{ALG:fast}
      {
            $\widehat{\underline{\mathbf X}} = {\rm fft}\left( {\underline{\mathbf X},[],3} \right)$;\\
            $\widehat{\underline{\mathbf Y}} = {\rm fft}\left( {\underline{\mathbf Y},[],3} \right)$;\\
            ${\underline{\mathbf G}}_1 =\widehat{\underline{\mathbf X}}(:,{\bf p},:) $;\\
${\underline{\mathbf G}}_2 =\widehat{\underline{\mathbf Y}}(:,{\bf p},:) $;\\
${\underline{\mathbf G}}_3=\widehat{\underline{\mathbf X}}({\bf s}_1,:,:) $;\\
${\underline{\mathbf G}}_4=\widehat{\underline{\mathbf Y}}({\bf s}_2,:,:) $;\\
\For{$i=1,2,\ldots,\lceil \frac{I_3+1}{2}\rceil$}
{                        
${\underline{\mathbf C}}_1\left( {:,:,i} \right) = \underline{\bf G}_1(:,:,i)\backslash (\widehat{\X}(:,:,i)/\underline{\bf G}_3({\bf s}_1,:,i))$;\\
${\underline{\mathbf C}}_2\left( {:,:,i} \right) = \underline{\bf G}_2(:,:,i)\backslash (\widehat{\X}(:,:,i)/\underline{\bf G}_4({\bf s}_2,:,i))$;\\
}
\For{$i=\lceil\frac{I_3+1}{2}\rceil+1,\ldots,I_3$}{
${\underline{\mathbf C}}_1\left( {:,:,i} \right)={\rm conj}({\underline{\mathbf C}}_1\left( {:,:,I_3-i+2} \right))$;\\
${\underline{\mathbf C}}_2\left( {:,:,i} \right)={\rm conj}({\underline{\mathbf C}}_2\left( {:,:,I_3-i+2} \right))$;
}
$\underline{\mathbf U}^{(1)} = {\rm ifft}\left( {{\underline{\mathbf C}}_1,[],3} \right)$;\\
$\underline{\mathbf U}^{(2)} = {\rm ifft}\left( {{\underline{\mathbf C}}_2,[],3} \right)$;
       	}       	
\end{algorithm}
In the subsequent discussion, and akin to \cite{gidisu2022generalized}, we establish the relationship between the GTCUR approximations of the tensor pairs $(\cX,\cY)$ and the TCUR approximation of the tensors $\cX*\cY^{-1}$ and $\cX*\cY^{\dag}$  in Theorem \ref{thm1}.

\begin{thm}\label{thm1}
Let two data tensors $\cX\in\mathbb{R}^{I_1\times I_2\times I_3}$ and $\cY\in\mathbb{R}^{I_4\times I_2 \times I_3}$ be given, then
\begin{itemize}
    \item a) If $\cY$ is invertible, then the selected lateral and horizontal slice indices obtained based on the TCUR approximation of the data tensor $\cX*\cY^{-1}$ are identical to the index vectors ${\bf s}_1$ and ${\bf s}_2$ computed from the GTCUR approximation of the tensor pairs $(\cX,\cY)$.

    \item b) As a special case, if $\cY=\underline{\bf I}$, then the GTCUR approximation of the tensor pairs $(\cX,\cY)$ is the same as the TCUR approximation of the data tensor $\cX$.

    \item c) If $\cY$ is noninvertible but of full tubal rank, then the selected lateral and horizontal slice indices obtained via the TCUR approximation of the data tensor $\cX*\cY^{\dag}$ are identical to the index vectors ${\bf s}_1$ and ${\bf s}_2$ computed via the GTCUR approximation of the tensor pairs $(\cX,\cY)$.
\end{itemize}

\end{thm}
\begin{proof}
a) Let the GTSVD of the tensor pairs $(\cX,\cY)$ be 
\begin{eqnarray}\label{gtscur_1}
\underline{\bf X}&=&\underline{\bf U}*\underline{\bf C}*\underline{\bf Z}^{T},\\\label{gtscur_2}
\underline{\bf Y}&=&\underline{\bf V}*\underline{\bf S}*\underline{\bf Z}^{T},
\end{eqnarray}
so, $\cX*\cY^{-1}=\cU*(\underline{\bf C}*\underline{\bf S}^{-1})*\underline{\bf V}^T$ because $\cY^{-1}=\cZ^{-T}*\cS^{-1}*\cV^{T}$. This means that the indices of the selected lateral and horizontal slices obtained via the GTCUR approximation of the tensor pairs $(\cX,\cY)$ are the same as those obtained from the TCUR approximation of the tensor $\cX*\cY^{-1}$ as both algorithms use the data tensors $\cU$ and $\cV$ to select the horizontal slices and lateral slices, respectively. \\

b) If $\cY=\underline{\bf I}$, then from \eqref{gtscur_2}, we have $\cZ^T=\underline{\bf S}^{-1}*\cV^T$ and substituting it in \eqref{gtscur_1}, we arrive at $\cX=\cU*(\underline{\bf C}*\underline{\bf S}^{-1})*\underline{\bf V}^T$. Considering the same explanations mentioned for part a, the proof of this part is completed. \\

c) Since the data tensor $\cY$ is of full tubal rank, its Moore-Penrose pseudoinverse can be represented as $\cY^{\dag}=\cZ^{-T}*\underline{\bf S}^{-1}*\cV^T$. So, the result is achieved using the same explanation as parts a and b.  
\end{proof}

\section{Tensor CUR approximation for tensor triplets based on the t-product}\label{Sec:PTACT}

\begin{figure}
\centering    \includegraphics[width=.8\linewidth]{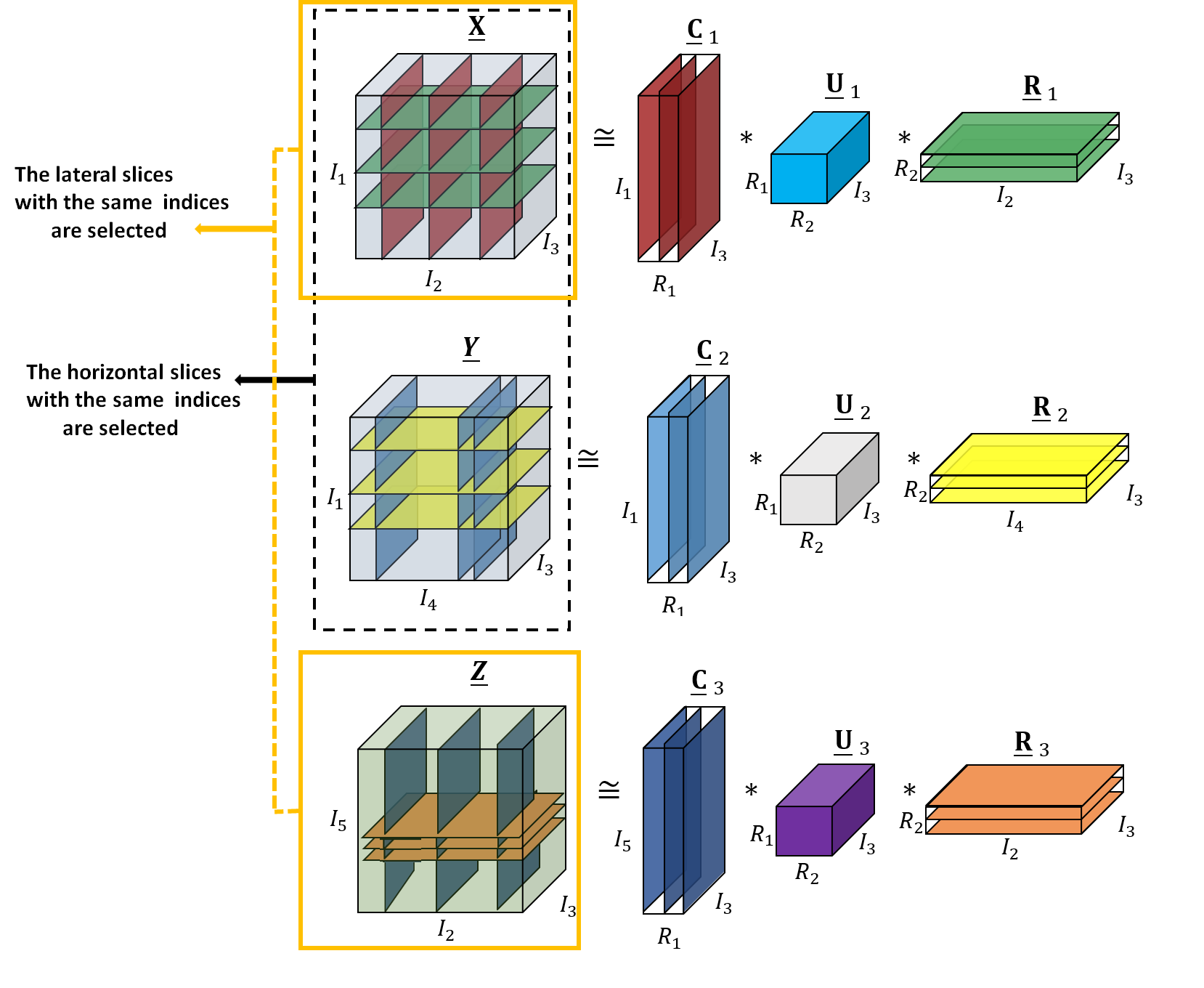}
    \caption{Visualization of the generalized tensor CUR (GTCUR) approximation for tensor triples $(\cX,\cY,\cZ)$. The indices of the selected horizontal slices of the tensors $\cX$ and $\cY$ are identical, but not those for the lateral slices. Also, the indices of the selected lateral slices of $\cX$ and $\cZ$ are the same, but not those for the horizontal slices.}\label{fig:CPj}
\end{figure}
The MCUR approximation for matrix triplets was recently introduced in \cite{gidisu2022restricted} as an extension of the MCUR concept. The approach involves utilizing the restricted SVD (RSVD) \cite{zha1991restricted} to choose the indices of both columns and rows. The RSVD decomposes a matrix in relation to two other matrices, effectively operating on three matrices at once. This concept can be similarly extended to tensors using the t-SVD, which we refer to as the GTCUR approximation for tensor triplets. Below, we provide the formal definition of this GTCUR approximation.

\begin{defn}
Let $\cX\in\mathbb{R}^{I_1\times I_2\times I_3}$, $\cY\in\mathbb{R}^{I_1\times I_4\times I_3}$, and $\cY\in\mathbb{R}^{I_5\times I_2\times I_3}$ be given. The GTCUR approximation of the tensor triplets $(\cX,\cY,\cZ)$ with the tubal rank $R$ is defined as
\begin{eqnarray}\label{e_1}
\cX &\approx &{\underline{\bf C}}_1*\cU_1*\cR_1 =(\cX*\cP)*\cU_1*({\cS}^T *\cX), \\
\cY &\approx &{\underline{\bf C}}_2*\cU_2*\cR_2 =(\cY*\widehat{\cP})*\cU_2 *({\cS}^T *\cY),\\\label{e_3}
\cZ &\approx&{\underline{\bf C}}_3*\cU_3*\cR_3 =(\cZ*\cP)*\cU_3 *(\widehat{\cS}^T *\cZ),
\end{eqnarray}
where $\underline{\bf S}\in\mathbb{R}^{I_1\times R\times I_3},\,\widehat{\underline{\bf S}}\in\mathbb{R}^{I_1\times R\times I_3},\,\cP\in\mathbb{R}^{I_2\times R\times I_3}$ and $\widehat{\cP}\in\mathbb{R}^{I_4\times R\times I_3},\,R\leq\min(I_1,I_2,I_4,I_5)$ are the index selection tensors yielded by sampling some lateral/horizontal slices of the identity tensor according to the sampled lateral/horizontal slices, which are sampled from the data tensors $\cX,\cY,$ and $\cZ$. Alternative representations for \eqref{e_1}-\eqref{e_3} are
\begin{eqnarray*}
\cX&\approx& \cX(:,{\bf p},:)*\underline{\U}_{1}*\cX({\bf s},:,:),\\
\cY&\approx&\cY(:,{\bf p}_1,:)*\underline{\U}_{2}*\cY({\bf s},:,:),\\
\cZ&\approx&\cZ(:,{\bf p},:)*\underline{\U}_{3}*\cZ({\bf s}_1,:,:),
\end{eqnarray*}
where ${\bf p},\,{\bf s},\,{\bf s}_1,\,{\bf p}_1$
are the index sets corresponding to the  selection tensors $\underline{\bf S}\in\mathbb{R}^{I_1\times R\times I_3},\,\widehat{\underline{\bf S}}\in\mathbb{R}^{I_1\times R\times I_3},\,\cP\in\mathbb{R}^{I_2\times R\times I_3}$ and $\widehat{\cP}\in\mathbb{R}^{I_4\times R\times I_3}$.
\end{defn}
The intuition behind the GTCUR approximation for tensor triplets also comes from the tensor-based restricted t-SVD (t-RSVD) which is a direct generalization of the classical RSVD \cite{zha1991restricted} to tensors and can be stated as follows. Let $\cX\in\mathbb{R}^{I_1\times I_2\times I_3}$, $\cY\in\mathbb{R}^{I_1\times I_4\times I_3}$, and $\cZ\in\mathbb{R}^{I_5\times I_2\times I_3}$ be given. Generalizing the RSVD to tensors based on the t-product guarantees the following decompositions
\begin{equation}\label{rsvd}
\cX=\cL*\cD_1*\cW^T,\quad \cY=\cL*\cD_2*\cU^T,\quad \cZ=\cV*\cD_3*\cW^T,
\end{equation}
where $\cU\in\mathbb{R}^{I_2\times I_2\times I_3},\,\cV\in\mathbb{R}^{I_5\times I_5\times I_3}$ are orthogonal and $\cL\in\mathbb{R}^{I_1\times I_1\times I_3},\,\cW\in\mathbb{R}^{I_2\times I_2\times I_3}$ are nonsingular. Also, $\cD_1\in\mathbb{R}^{I_1\times I_2\times I_3},\,\cD_2\in\mathbb{R}^{I_1\times I_4\times I_3},\,\cD_3\in\mathbb{R}^{I_5\times I_2\times I_3}$ have quasi-diagonal frontal slices, which means that if we remove their zero columns and zero rows, they become diagonal. The computation of the RSVD of the matrices can be performed using a double GSVD as described in \cite{zha1991restricted,de1991restricted}. Considering this fact, the tensor RSVD (t-RSVD) can be computed through the RSVD of the frontal slices of the underlying data tensors in the Fourier domain. We can indeed modify Algorithm \ref{ALG:GTSVD} to be used for computing t-RSVD. More precisely, the fft of all three input tensors are computed, then the RSVD is applied for the corresponding frontal slices of the tensors in the Fourier domain. Finally, the inverse fft is applied to get the factor tensors.

The computational process for the GTCUR approximation of tensor triples is outlined in Algorithm \ref{ALG:RtCUR}. Lines 6-8 can be efficiently executed in the Fourier domain, and similar algorithms, akin to Algorithm  \ref{ALG:DEIM_genten}, can be devised for this computation. The t-RSVD of the tensor triples $(\cX,\cY,\cZ)$ yields the common tensor factors $\cL,\,\cW$, along with additional tensor factors $\cU$ and $\cV$. These bases facilitate the sampling of lateral and horizontal slices from the tensor triples$(\cX,\cY,\cZ)$. Since the common tensor factors $\cL$ and $\cW$ are available, the indices used to sample the lateral slices of the data tensors $\cX$ and $\cZ$ are the same, while the indices for selecting horizontal slices of tensors $\cX$ and $\cY$ are also identical. However, the tensor bases $\cU$ and $\cV$ are utilized to sample the lateral slice indices of the tensor $\cY$ and the horizontal slice indices of the tensor $\cZ$, respectively. This approach is visually represented in Figure \ref{fig:CPj}. Essentially, the goal is to compute the t-RSVD of the tensor triples$(\cX,\cY,\cZ)$ to identify the corresponding tensor factors $\cU,\,\cV,\,\cZ,$ and $\cW$, which are then employed to determine the indices for selecting horizontal and lateral slices using the TDEIM algorithm. It is important to note that computing the t-RSVD can be resource-intensive for large-scale data tensors; therefore, the fast randomized GSVD algorithm proposed in \cite{ahmadi2025randomized} can be utilized to perform the double GSVD necessary for calculating the t-RSVD of the tensor triples. The relationship between the GTCUR approximation for tensor triples and its counterparts for tensor pairs and TCUR approximations is discussed in the following theorem.

\begin{thm} Let $\cX\in\mathbb{R}^{I_1\times I_2\times I_3}$, $\cY\in\mathbb{R}^{I_1\times I_4\times I_3}$, and $\cZ\in\mathbb{R}^{I_5\times I_2\times I_3}$ be given, then

\begin{itemize}
\item a) If $\cY$ and $\cZ$ are nonsingular tensors, then the selected lateral and horizontal slice indices from the TCUR approximation of $\cY^{-1}*\cX*\cZ^{-1}$ are identical to the index vectors ${\bf p}_1$ and ${\bf s}_1$, respectively,  obtained from a GTCUR approximation of the tensor triples $(\cX,\cY,\cZ)$.

\item  b) In the particular case where $\cY = \underline{\I}$ and $\cZ = \underline{\I}$, the GTCUR approximation of $(\cX,\underline{\bf I},\underline{\bf I})$ coincides with the TCUR approximation of $\cX$.

\item c) For a special choice of $\cY=\underline{\bf I}$, a GTCUR approximation of $(\cX,\underline{\bf I},\cZ)$ coincides with the GTCUR approximation of $(\cX,\cZ)$. Similar results can be stated for $\cZ = \underline{\bf I}$.
\end{itemize}
\end{thm}

\begin{proof}
a) From \ref{rsvd}, it is not difficult to see that $\cY^{-1}=\cU*{\cD}_2^{-1}*\cL^{-1},\,\cZ^{-1}=\cW*\cD_3^{-1}*\cV^T$, so we have 
\begin{eqnarray}
\nonumber
\cY^{-1}*\cX*\cZ^{-1}=(\cU*{\cD}_2^{-1}*\cL^{-1})*(\cL*\cD_1*\cW^T)*(\cW*\cD_3^{-1}*\cV^T)=\\\label{prf}
\cU*(\cD_2^{-1}\cD_1\cD_3^{-1})*\cV^T.
\end{eqnarray}
This means that the selected lateral and horizontal slice indices of $\cY^{-1}*\cX*\cZ^{-1}$ using the TCUR approximation are the same as the index vectors ${\bf p}_1$ and ${\bf s}_1$, respectively, obtained from a GTCUR approximation of $(\cX,\cY,\cZ).$\\

b) If $\cY=\underline{\I}$ and $\cZ=\underline{\I}$, then $\cL=\cU*{\cD}_2^{-1}$ and $\cW^T=\cD_3^{-1}*\cV^T$. Substituting them in the first part of \eqref{rsvd}, we have $\cX=\cU*(\cD_2^{-1}*\cD_1*\cD_3^{-1})*\cV^T$. This indicates that both algorithms utilize the same tensors $\cU$ and $\cV$  for selecting the horizontal and lateral slices, respectively. This completes the proof.

c) If $\cY=\underline{\I}$, then from the second part of \eqref{rsvd}, we have $\cL=\cU*\cD_2^{-1}$. If we substitute this in the first part of \eqref{rsvd}, we arrive at 
\begin{eqnarray}\label{rs_1}
\cX&=&\cU*(\cD_2^{-1}*\cD_1)*\cW^T,\\\label{rs_2}
\cZ&=&\cV*\cD_3*\cW^T.  
\end{eqnarray}
It is evident that equations \eqref{rs_1} and \eqref{rs_2} demonstrate that the GTCUR approximation of the tensor pairs $(\cX,\cZ)$ yields identical indices for both horizontal and lateral slice sampling as the GTCUR approximation of the tensor triples $(\cX,\underline{\bf I},\cZ)$. This completes the proof.

\end{proof}

\RestyleAlgo{ruled}
\LinesNumbered
\begin{algorithm}
\SetKwInOut{Input}{Input}
\SetKwInOut{Output}{Output}\Input{Three data tensors $\cX\in\mathbb{R}^{I_1\times I_2\times I_3},\,\cY\in\mathbb{R}^{I_1\times I_4\times I_3},$ and $\cZ\in\mathbb{R}^{I_5\times I_2\times I_3}$; and a target tubal rank $R$} 
\Output{A tubal rank-$R$ GTCUR approximation of tensor triplets $(\cX,\cY,\cZ)$ as 
\begin{eqnarray*}
\cX&\approx& \cX(:,{\bf p},:)*\underline{\U}_{1}*\cX({\bf s},:,:),\\
\cY&\approx&\cY(:,{\bf p}_1,:)*\underline{\U}_{2}*\cY({\bf s},:,:)\\
\cZ&\approx&\cZ(:,{\bf p},:)*\underline{\U}_{3}*\cZ({\bf s}_1,:,:)
\end{eqnarray*}
}
\caption{GTCUR approximation of  tensor triplets}\label{ALG:RtCUR}
      {
Compute the t-RSVD (either deterministic or randomized version) of $(\cX,\cY,\cZ)$ to obtain $\cW,\,\cZ,\cU,\cV$;\\
${\bf p}\leftarrow$ Apply Algorithm \ref{ALG:DEIM_ten} to $\cW$;\\
${\bf s}\leftarrow$ Apply Algorithm \ref{ALG:DEIM_ten} to $\cZ$;\\
${\bf p}_1\leftarrow$ Apply Algorithm \ref{ALG:DEIM_ten} to $\cU$;\\
${\bf s}_1\leftarrow$ Apply Algorithm \ref{ALG:DEIM_ten} to $\cV$;\\
${\cU}_{1}=\cX(:,{\bf p},:)^{\dag}*\cX*\cX({\bf s},:,:)^{\dag}$;\\
${\cU}_{2}=\cY(:,{\bf p}_1,:)^{\dag}*\cY*\cY({\bf s},:,:)^{\dag}$;\\
${\cU}_{3}=\cZ(:,{\bf p},:)^{\dag}*\cZ*\cZ({\bf s}_1,:,:)^{\dag}$;\\
       	}       	
\end{algorithm}

\section{Numerical results}\label{Sec:num}
This section demonstrates the effectiveness and practical applicability of the proposed generalized tensor CUR (GTCUR) algorithms for both tensor pairs and triplets. All simulations were implemented in MATLAB, utilizing functions from the Tensor-tensor-product Toolbox\footnote{\url{https://github.com/canyilu/Tensor-tensor-product-toolbox.}}.. The computations were performed on a laptop equipped with an Intel® Core™ i7-5600U processor (2.60 GHz) and 8GB RAM. We evaluate the algorithm performance using two key metrics: Relative error and Peak signal-to-noise ratio (PSNR).

The relative error between a reference tensor $\underline{\bf X}$ and its approximation $\underline{\bf Y}$ is defined as:
\begin{equation}
    \text{Relative Error} = \frac{\|\underline{\bf X} - \underline{\bf Y}\|_F}{\|\underline{\bf X}\|_F}.
\end{equation}
The PSNR for two images $\underline{\bf X}$ and $\underline{Y}$ is computed as:
\begin{equation}
    \text{PSNR}(\underline{\mathbf{X}},\underline{\mathbf{Y}}) = 10 \cdot \log_{10} \left( \frac{255^2}{\text{MSE}(\underline{\mathbf{X}},\underline{\mathbf{Y}})} \right),
\end{equation}
where the Mean Squared Error (MSE) is calculated by:
\begin{equation}
    \text{MSE}(\underline{\mathbf{X}},\underline{\mathbf{Y}}) = \frac{1}{N} \sum_{i=1}^N ({\bf x}_i - {\bf y}_i)^2,
\end{equation}
with $N$ being the total number of pixels and ${\bf x}_i$ and ${\bf y}_i$ representing the $i$-th elements of vectorized images ${\bf x}=\underline{\mathbf{X}}(:)$ and ${\bf y}=\underline{\mathbf{Y}}(:)$ respectively.
    
\begin{exa}\label{exa_1}
({\bf Random tensors})
In this experiment, we generate random third-order tensors with entries drawn independently and identically from a standard Gaussian distribution ($N(0,1)$). Specifically, we construct two random tensors of size $100\times 100\times 100$ with tubal rank 40 using the following procedure:
\begin{enumerate}
    \item Generate two random tensors:
    \begin{itemize}
        \item Tensor $\underline{\bf A}$: $100\times 40\times 100$.
        \item Tensor $\underline{\bf B}$: $40\times 100\times 100$.
    \end{itemize}
    \item Compute their t-product to obtain the final tensors: $\underline{\bf X} = \underline{\bf A}* \underline{\bf B}$.
\end{enumerate}
This construction method guarantees that the resulting tensors maintain the desired tubal rank of 40 while preserving the Gaussian distribution properties of their elements. We subsequently applied the GTCUR method to both data tensors, with the reconstruction results presented in Figure \ref{fig:gtcur_1}.

\begin{figure}
     \centering
     \subfigure[Relative error vs.\ tensor rank]{
         \centering
         \includegraphics[width=0.45\linewidth]{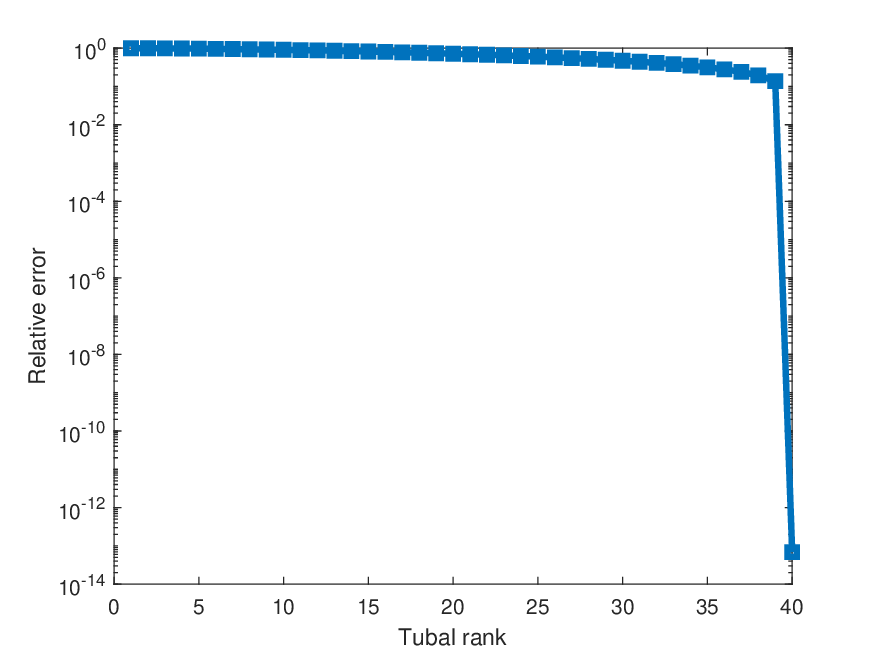}
     }
     \hfill
     \subfigure[Relative error vs.\ tensor rank]{
         \centering
         \includegraphics[width=.45\linewidth]{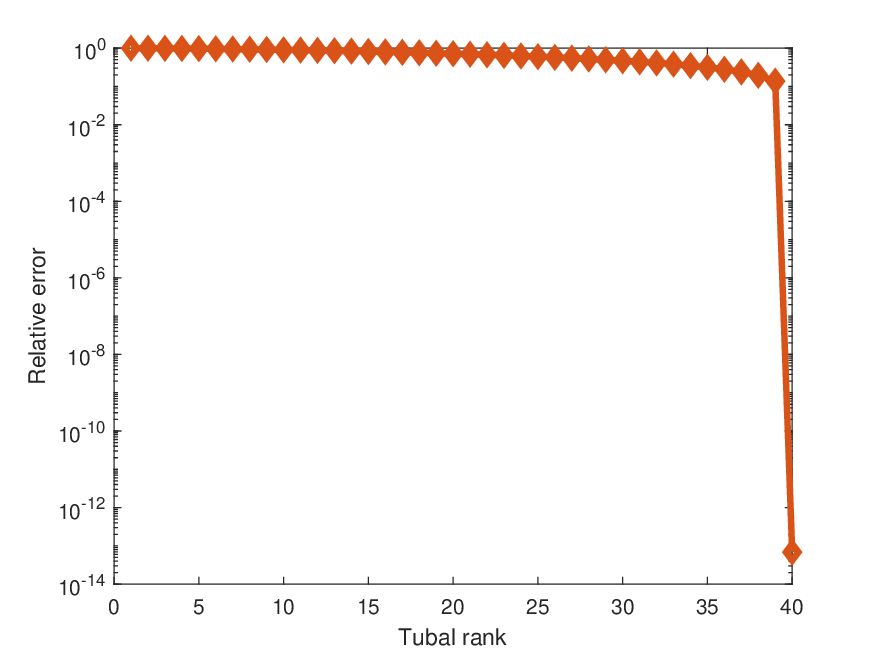}
     }
             \caption{The relative error history of GTCUR approximation for tensor pairs, left is for the tensor $\underline{\bf X}$ and right is for the tensor $\underline{\bf Y}$.}
        \label{fig:gtcur_1}
\end{figure}
To evaluate the algorithm for the tensor triples case, we generated three tensors $\underline{\bf X},\,\underline{\bf Y},\underline{\bf Z}$ of size $100\times 100\times 100$ and tubal rank $40$. The GTCUR approximation for tensor triplets of tubal rank $40$ was applied to the mentioned tensors. Like the case with tensor pairs, the GTCUR approximation for tensor triples achieves exact results once the number of sampled lateral or horizontal slices reaches 40. The experiments indicate that the GTCUR approximation is effective for both tensor pairs and tensor triples when applied to random data tensors.

\end{exa}

\begin{exa}\label{exa_2} ({\bf Image approximation})
In this example, we examine the GTCUR approximation for tensor pairs and tensor triplets. We begin with pairs of tensors, using the \textit{"Lena"} and \textit{"Peppers"} images (both of size $256\times 256\times 3$), see Figure \ref{fig:samples}. We apply the proposed GTCUR method (tensor pairs case) to sample lateral and horizontal slices. Notably, the indices of the selected lateral slices remain consistent (see Figure \ref{fig:pic2}, bottom). The reconstructed images, obtained using $R = 60$ lateral/horizontal slices, are displayed in Figure \ref{fig:pic2}. 

\begin{figure}
\centering    \includegraphics[width=.6\linewidth]{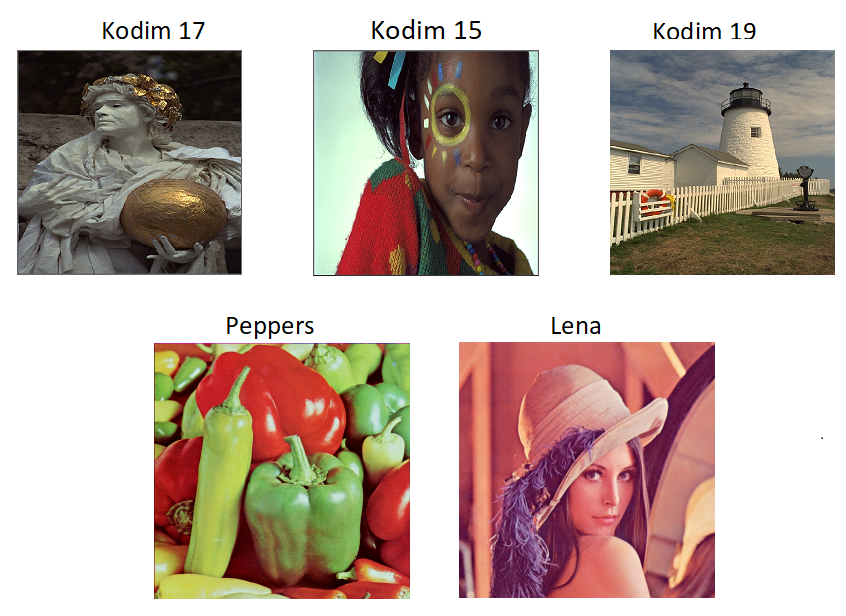}
    \caption{The sample images utilized in our calculations.}\label{fig:samples}
\end{figure}

Next, we evaluate the method on three images \textit{"kodim17"}, \textit{"kodim15"}, and \textit{"kodim19"} from the Kodak dataset\footnote{\url{https://r0k.us/graphics/kodak/}}, , see Figure \ref{fig:samples}.. Originally, \textit{"kodim17"}, and \textit{"kodim19"}   are of size $768\times 512\times 3$, while \textit{"kodim15"} is of size $512\times 768\times 3$. We resize them to $512\times 512\times 3$ for the computation to be manageable. We then applied the GTCUR method (tensor triplets case) under the same conditions (R = 120 slices). The reconstructed images and their corresponding PSNR values are shown in Figure \ref{fig:pic2} (upper). 

The results demonstrate that the GTCUR method delivers high-quality reconstructions for both image pairs and image triplets. 

\begin{figure}
\centering    \includegraphics[width=0.7\linewidth]{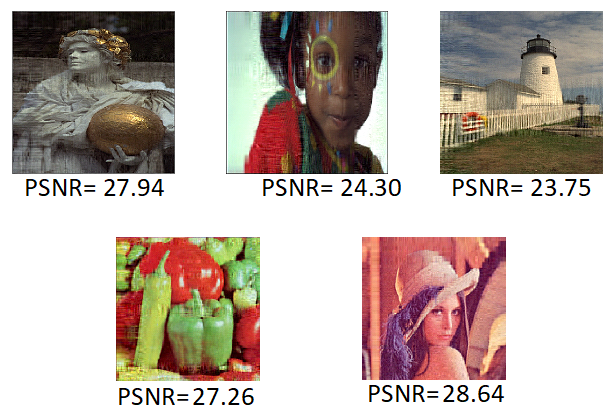}
    \caption{The images reconstructed by the proposed GTCUR method, for both image pairs and image triplets.}\label{fig:pic2}
\end{figure}
\end{exa}

\section{Conclusion}\label{Sec:Con}
In this note, we demonstrated how the tensor CUR (TCUR) approximation can be adapted for tensor pairs and tensor triplets. We utilize the tensor Discrete Interpolatory Empirical Method (TDEIM) to extend the TCUR approximation to these tensor configurations, resulting in what we call the generalized TCUR (GTCUR) method. We established links between certain specific cases of GTCUR and the traditional TCUR approximations.Additionally, we introduced efficient algorithms for computing the GTCUR approximation for both tensor pairs and tensor triplets. Our ongoing research focuses on the theoretical and numerical aspects of the GTCUR methods for tensor pairs and triplets in practical applications.

\section*{Acknowledgment}
The authors would like to thank the reviewer for constructive and useful comment, which improved the quality of the paper significantly. 

\section{Conflict of Interest Statement}
 The authors declare that they have no
 conflict of interest with anything.

\bibliographystyle{elsarticle-num} 
\bibliography{cas-refs}

@article{tarzanagh2018fast,
  title={Fast randomized algorithms for t-product based tensor operations and decompositions with applications to imaging data},
  author={Tarzanagh, Davoud Ataee and Michailidis, George},
  journal={SIAM Journal on Imaging Sciences},
  volume={11},
  number={4},
  pages={2629--2664},
  year={2018},
  publisher={SIAM}
}

@article{kilmer2011factorization,
  title={Factorization strategies for third-order tensors},
  author={Kilmer, Misha E and Martin, Carla D},
  journal={Linear Algebra and its Applications},
  volume={435},
  number={3},
  pages={641--658},
  year={2011},
  publisher={Elsevier}
}

@article{kilmer2013third,
  title={Third-order tensors as operators on matrices: A theoretical and computational framework with applications in imaging},
  author={Kilmer, Misha E and Braman, Karen and Hao, Ning and Hoover, Randy C},
  journal={SIAM Journal on Matrix Analysis and Applications},
  volume={34},
  number={1},
  pages={148--172},
  year={2013},
  publisher={SIAM}
}

@article{lu2019tensor,
  title={Tensor robust principal component analysis with a new tensor nuclear norm},
  author={Lu, Canyi and Feng, Jiashi and Chen, Yudong and Liu, Wei and Lin, Zhouchen and Yan, Shuicheng},
  journal={IEEE transactions on pattern analysis and machine intelligence},
  volume={42},
  number={4},
  pages={925--938},
  year={2019},
  publisher={IEEE}
}

@article{oseledets2010tt,
  title={{TT}-cross approximation for multidimensional arrays},
  author={Oseledets, Ivan and Tyrtyshnikov, Eugene},
  journal={Linear Algebra and its Applications},
  volume={432},
  number={1},
  pages={70--88},
  year={2010},
  publisher={Elsevier}
}

@article{oseledets2008tucker,
  title={{T}ucker dimensionality reduction of three-dimensional arrays in linear time},
  author={Oseledets, Ivan V and Savostianov, DV and Tyrtyshnikov, Eugene E},
  journal={SIAM Journal on Matrix Analysis and Applications},
  volume={30},
  number={3},
  pages={939--956},
  year={2008},
  publisher={SIAM}
}

@article{tucker1964extension,
  title={The extension of factor analysis to three-dimensional matrices},
  author={Tucker, Ledyard R and others},
  journal={Contributions to mathematical psychology},
  volume={110119},
  year={1964},
  publisher={New York: Holt, Rinehardt \& Winston}
}

@article{tucker1966some,
  title={Some mathematical notes on three-mode factor analysis},
  author={Tucker, Ledyard R},
  journal={Psychometrika},
  volume={31},
  number={3},
  pages={279--311},
  year={1966},
  publisher={Springer}
}

@article{hitchcock1927expression,
  title={The expression of a tensor or a polyadic as a sum of products},
  author={Hitchcock, Frank L},
  journal={Journal of Mathematics and Physics},
  volume={6},
  number={1-4},
  pages={164--189},
  year={1927},
  publisher={Wiley Online Library}
}

@article{hitchcock1928multiple,
  title={Multiple invariants and generalized rank of a p-way matrix or tensor},
  author={Hitchcock, Frank L},
  journal={Journal of Mathematics and Physics},
  volume={7},
  number={1-4},
  pages={39--79},
  year={1928},
  publisher={Wiley Online Library}
}

@article{espig2012note,
  title={A note on tensor chain approximation},
  author={Espig, Mike and Naraparaju, Kishore Kumar and Schneider, Jan},
  journal={Computing and Visualization in Science},
  volume={15},
  number={6},
  pages={331--344},
  year={2012},
  publisher={Springer}
}

@incollection{goreinov2010find,
  title={How to find a good submatrix},
  author={Goreinov, Sergei A and Oseledets, Ivan V and Savostyanov, Dimitry V and Tyrtyshnikov, Eugene E and Zamarashkin, Nikolay L},
  booktitle={Matrix Methods: Theory, Algorithms And Applications: Dedicated to the Memory of Gene Golub},
  pages={247--256},
  year={2010},
  publisher={World Scientific}
}

@article{goreinov1997theory,
  title={A theory of pseudoskeleton approximations},
  author={Goreinov, Sergei A and Tyrtyshnikov, Eugene E and Zamarashkin, Nickolai L},
  journal={Linear algebra and its applications},
  volume={261},
  number={1-3},
  pages={1--21},
  year={1997},
  publisher={Elsevier}
}

@article{van1976generalizing,
  title={Generalizing the singular value decomposition},
  author={Van Loan, Charles F},
  journal={SIAM Journal on numerical Analysis},
  volume={13},
  number={1},
  pages={76--83},
  year={1976},
  publisher={SIAM}
}

@article{paige1981towards,
  title={Towards a generalized singular value decomposition},
  author={Paige, Christopher C and Saunders, Michael A},
  journal={SIAM Journal on Numerical Analysis},
  volume={18},
  number={3},
  pages={398--405},
  year={1981},
  publisher={SIAM}
}

@article{gidisu2022generalized,
  title={A Generalized {CUR} decomposition for matrix pairs},
  author={Gidisu, Perfect Y and Hochstenbach, Michiel E},
  journal={SIAM Journal on Mathematics of Data Science},
  volume={4},
  number={1},
  pages={386--409},
  year={2022},
  publisher={SIAM}
}

@article{he2021generalized,
  title={Generalized Singular Value Decompositions for Tensors and Their Applications.},
  author={He, Zhuo-Heng and Ng, Michael K and Zeng, Chao},
  journal={Numerical Mathematics: Theory, Methods \& Applications},
  volume={14},
  number={3},
  year={2021}
}

@article{zhang2021cs,
  title={{CS} decomposition and {GSVD} for tensors based on the T-product},
  author={Zhang, Yating and Guo, Xiaoxia and Xie, Pengpeng and Cao, Zhengbang},
  journal={arXiv preprint arXiv:2106.16073},
  year={2021}
}

@article{sorensen2016deim,
  title={A {DEIM} induced {CUR} factorization},
  author={Sorensen, Danny C and Embree, Mark},
  journal={SIAM Journal on Scientific Computing},
  volume={38},
  number={3},
  pages={A1454--A1482},
  year={2016},
  publisher={SIAM}
}

@article{gidisu2022restricted,
  title={A Restricted {SVD} type {CUR} Decomposition for Matrix Triplets},
  author={Gidisu, Perfect Y and Hochstenbach, Michiel E},
  journal={arXiv preprint arXiv:2204.02113},
  year={2022}
}

@incollection{gidisu2022hybrid,
  title={A hybrid {DEIM} and leverage scores based method for {CUR} index selection},
  author={Gidisu, Perfect Y and Hochstenbach, Michiel E},
  booktitle={Progress in Industrial Mathematics at ECMI 2021},
  pages={147--153},
  year={2022},
  publisher={Springer}
}

@article{zha1991restricted,
  title={The restricted singular value decomposition of matrix triplets},
  author={Zha, Hongyuan},
  journal={SIAM journal on matrix analysis and applications},
  volume={12},
  number={1},
  pages={172--194},
  year={1991},
  publisher={SIAM}
}

@article{mahoney2011randomized,
  title={Randomized algorithms for matrices and data},
  author={Mahoney, Michael W and others},
  journal={Foundations and Trends{\textregistered} in Machine Learning},
  volume={3},
  number={2},
  pages={123--224},
  year={2011},
  publisher={Now Publishers, Inc.}
}

@article{de1991restricted,
  title={The restricted singular value decomposition: properties and applications},
  author={De Moor, Bart LR and Golub, Gene H},
  journal={SIAM Journal on Matrix Analysis and Applications},
  volume={12},
  number={3},
  pages={401--425},
  year={1991},
  publisher={SIAM}
}

@article{drineas2008relative,
  title={Relative-error {CUR} matrix decompositions},
  author={Drineas, Petros and Mahoney, Michael W and Muthukrishnan, Shan},
  journal={SIAM Journal on Matrix Analysis and Applications},
  volume={30},
  number={2},
  pages={844--881},
  year={2008},
  publisher={SIAM}
}

@article{ahmadi2024randomized,
  title={A randomized algorithm for tensor singular value decomposition using an arbitrary number of passes},
  author={Ahmadi-Asl, Salman and Phan, Anh-Huy and Cichocki, Andrzej},
  journal={Journal of Scientific Computing},
  volume={98},
  number={1},
  pages={23},
  year={2024},
  publisher={Springer}
}

@article{ahmadi2022efficient,
  title={An Efficient Randomized Fixed-Precision Algorithm for Tensor Singular Value Decomposition},
  author={Ahmadi-Asl, Salman},
  journal={Communications on Applied Mathematics and Computation},
  pages={1--20},
  year={2022},
  publisher={Springer}
}

@article{zhang2018randomized,
  title={A randomized tensor singular value decomposition based on the t-product},
  author={Zhang, Jiani and Saibaba, Arvind K and Kilmer, Misha E and Aeron, Shuchin},
  journal={Numerical Linear Algebra with Applications},
  volume={25},
  number={5},
  pages={e2179},
  year={2018},
  publisher={Wiley Online Library}
}

@article{caiafa2010generalizing,
  title={Generalizing the column--row matrix decomposition to multi-way arrays},
  author={Caiafa, Cesar F and Cichocki, Andrzej},
  journal={Linear Algebra and its Applications},
  volume={433},
  number={3},
  pages={557--573},
  year={2010},
  publisher={Elsevier}
}

@article{hao2013facial,
  title={Facial recognition using tensor-tensor decompositions},
  author={Hao, Ning and Kilmer, Misha E and Braman, Karen and Hoover, Randy C},
  journal={SIAM Journal on Imaging Sciences},
  volume={6},
  number={1},
  pages={437--463},
  year={2013},
  publisher={SIAM}
}

@article{kernfeld2015tensor,
  title={Tensor--tensor products with invertible linear transforms},
  author={Kernfeld, Eric and Kilmer, Misha and Aeron, Shuchin},
  journal={Linear Algebra and its Applications},
  volume={485},
  pages={545--570},
  year={2015},
  publisher={Elsevier}
}

@article{jiang2020framelet,
  title={Framelet representation of tensor nuclear norm for third-order tensor completion},
  author={Jiang, Tai-Xiang and Ng, Michael K and Zhao, Xi-Le and Huang, Ting-Zhu},
  journal={IEEE Transactions on Image Processing},
  volume={29},
  pages={7233--7244},
  year={2020},
  publisher={IEEE}
}

@article{li2022nonlinear,
  title={Nonlinear transform induced tensor nuclear norm for tensor completion},
  author={Li, Ben-Zheng and Zhao, Xi-Le and Ji, Teng-Yu and Zhang, Xiong-Jun and Huang, Ting-Zhu},
  journal={Journal of Scientific Computing},
  volume={92},
  number={3},
  pages={83},
  year={2022},
  publisher={Springer}
}

@article{song2020robust,
  title={Robust tensor completion using transformed tensor singular value decomposition},
  author={Song, Guangjing and Ng, Michael K and Zhang, Xiongjun},
  journal={Numerical Linear Algebra with Applications},
  volume={27},
  number={3},
  pages={e2299},
  year={2020},
  publisher={Wiley Online Library}
}

@article{elden2019solving,
  title={Solving bilinear tensor least squares problems and application to Hammerstein identification},
  author={Eld{\'e}n, Lars and Ahmadi-Asl, Salman},
  journal={Numerical Linear Algebra with Applications},
  volume={26},
  number={2},
  pages={e2226},
  year={2019},
  publisher={Wiley Online Library}
}

@article{ahmadi2024robust,
  title={Robust low tubal rank tensor recovery using discrete empirical interpolation method with optimized slice/feature selection},
  author={Ahmadi-Asl, Salman and Phan, Anh-Huy and Caiafa, Cesar F and Cichocki, Andrzej},
  journal={Advances in Computational Mathematics},
  volume={50},
  number={2},
  pages={23},
  year={2024},
  publisher={Springer}
}

@article{asante2021matrix,
  title={Matrix and tensor completion using tensor ring decomposition with sparse representation},
  author={Asante-Mensah, Maame G and Ahmadi-Asl, Salman and Cichocki, Andrzej},
  journal={Machine Learning: Science and Technology},
  volume={2},
  number={3},
  pages={035008},
  year={2021},
  publisher={IOP Publishing}
}

@article{asante2023image,
  title={Image reconstruction using superpixel clustering and tensor completion},
  author={Asante-Mensah, Maame G and Phan, Anh Huy and Ahmadi-Asl, Salman and Al Aghbari, Zaher and Cichocki, Andrzej},
  journal={Signal Processing},
  volume={212},
  pages={109158},
  year={2023},
  publisher={Elsevier}
}

@article{ahmadi2021cross,
  title={Cross tensor approximation methods for compression and dimensionality reduction},
  author={Ahmadi-Asl, Salman and Caiafa, Cesar F and Cichocki, Andrzej and Phan, Anh Huy and Tanaka, Toshihisa and Oseledets, Ivan and Wang, Jun},
  journal={IEEE Access},
  volume={9},
  pages={150809--150838},
  year={2021},
  publisher={IEEE}
}

@article{ahmadi2025randomized,
  title={Randomized Algorithms for Computing the Generalized Tensor SVD Based on the Tensor Product},
  author={Ahmadi-Asl, Salman and Rezaeian, Naeim and Ugwu, Ugochukwu O},
  journal={Communications on Applied Mathematics and Computation},
  pages={1--17},
  year={2025},
  publisher={Springer}
}
\end{document}